\documentclass[12pt]{article}
\usepackage[utf8]{inputenc}
\usepackage{graphicx}
\usepackage{color}
\usepackage{times}
\usepackage[hidelinks]{hyperref}
\usepackage{enumerate,latexsym}
\usepackage{latexsym}
\usepackage{amsmath,amssymb}
\usepackage{graphicx}

\usepackage{amsmath,amsthm,amsfonts,amssymb}
\usepackage{graphicx}
\usepackage{dsfont}

\makeatletter
\def\namedlabel#1#2{\begingroup
 #2%
 \def\@currentlabel{#2}%
 \phantomsection\label{#1}\endgroup
}
\makeatother

\theoremstyle{plain}
\newtheorem*{theorem*}{Theorem}
\newtheorem*{thmex*}{Theorem~\ref{example}}
\newtheorem*{thmasymp*}{Theorem~\ref{thmAsymp}}
\newtheorem{theorem}{Theorem}[section]

\newtheorem{case}[theorem]{Case}
\newtheorem{claim}[theorem]{Claim}

\newtheorem{corollary}[theorem]{Corollary}

\newtheorem{lemma}[theorem]{Lemma}

\newtheorem{remark}[theorem]{Remark}

\theoremstyle{definition}
\newtheorem{definition}[theorem]{Definition}

\newcommand{\R}{\mathbb{R}}

\newcommand{\esf}{\mathbb{S}}

\renewcommand{\l}{\lambda}

\newcommand{\N}{{\mathbb{N}}}
\newcommand{\sn}[1]{{\mathbb{S}^{#1}}}
\newcommand{\hn}[1]{{\mathbb{H}^{#1}}}

\newcommand{\G}{\Gamma}
\newcommand{\ben}{\begin{enumerate}}
\newcommand{\een}{\end{enumerate}}
\newcommand{\wt}{\widetilde}
\newcommand{\g}{\gamma}

\renewcommand{\a}{\alpha}

\newcommand{\cB}{{\mathcal B}}
\newcommand{\cC}{{\mathcal C}}

\newcommand{\rth}{\R^3}

\newcommand{\cF}{{\mathcal F}}

\newcommand{\cG}{{\mathcal G}}

\newcommand{\cS}{{\mathcal S}}
\newcommand{\cT}{{\mathcal T}}

\newcommand{\ve}{\varepsilon}

\newcommand{\ed}{\end{document}}
\renewcommand{\S}{\Sigma}

\newcommand{\ol}{\overline}
\newcommand{\wh}{\widehat}
\definecolor{rrr}{rgb}{.9,0,.1}

\definecolor{rr}{rgb}{.8,0,.3}

\usepackage{pdfsync}
\usepackage{fullpage}

\begin{document}

\title{Totally umbilic surfaces in hyperbolic 3-manifolds of finite volume}

\author{Colin~Adams
\and William H. Meeks III
\and Alvaro K. Ramos\thanks{The second and third
authors were partially supported
by CNPq - Brazil, grant no. 400966/2014-0.}.}

\maketitle

\begin{abstract}
We construct for every connected
surface $S$ of finite negative Euler characteristic and every $H \in [0,1)$,
a hyperbolic 3-manifold $N(S,H)$ of finite volume
and a proper, two-sided, totally umbilic
embedding $f\colon S\to N(S,H)$ with mean curvature $H$. Conversely, we prove
that a complete, totally umbilic surface with
mean curvature $H \in [0,1)$ embedded in a hyperbolic 3-manifold
of finite volume must be proper and have finite negative Euler characteristic.

\vspace{.15cm}

\vspace{.1cm}

\end{abstract}

\section{Introduction.}

In this manuscript we develop the theory of totally umbilic
surfaces in hyperbolic 3-manifolds
of finite volume; all spaces considered here are assumed
to be complete and connected and all surfaces in them
will be assumed to be complete, connected and embedded,
unless otherwise stated.

\begin{theorem}\label{proper1} Let $\S$ be a totally umbilic surface
with mean curvature $H_\S\geq0$
in a hyperbolic 3-manifold $N$ of finite volume.
Then:\ben
\item\label{itemprop} $\S$ is proper in $N$.
\item\label{itempos} $\S$ has positive Euler characteristic if and only if
$\S$ is a geodesic sphere. In particular,
$\S$ is not diffeomorphic to a plane or a projective plane.
\item\label{itemzero} $\S$ has zero Euler characteristic if and only if $N$ is non-compact
and $H_\S = 1$. In this case,
$\S$ is a flat torus or a flat Klein bottle that
is contained in some cusp end of $N$.
\item\label{itemneg} $\S$ has negative Euler characteristic if and only if
it has finite negative Euler characteristic if and only if $H_\S\in[0,1)$. Furthermore:
\ben
\item\label{itemnega} $\S$ has finite area
$A(\S) = \frac{2\pi}{H_\S^2-1}\chi(\S)$, where $\chi(\S)$ is the Euler characteristic of $\S$.
\item\label{itemnegb} If $H_\S>0$, then, for every $H\in (0,H_\S)$,
there is a
totally umbilic surface with mean curvature $H$ in the
ambient isotopy class of $\S$.
\item\label{itemnegc} If $H_\S>0$, then
there is a
totally geodesic surface $\S_0$ in the isotopy class of $\S$;
also, $\S$ is diffeomorphic to $\S_0$ if $\S_0$ is two-sided
and $\S$ is diffeomorphic to the two-sided
cover of $\S_0$ if $\S_0$ is one-sided.
\een
\een
\end{theorem}

The next theorem characterizes the admissible topological types
of totally umbilic surfaces in hyperbolic 3-manifolds of finite volume
with mean curvature in $[0,1)$. It is a direct
consequence of Theorem~\ref{proper1} and
Theorem~\ref{thmcmc}.

\begin{theorem}[Admissible Topology Theorem]\label{admiss} \
A surface $S$ appears topologically as a
totally umbilic surface with mean curvature $
H\in [0,1)$ in some hyperbolic 3-manifold of finite
volume if and only if $S$ has finite negative Euler characteristic.
\end{theorem}

Our construction of a hyperbolic 3-manifold of finite volume with a given admissible, two-sided,
totally geodesic surface depends on
the Switch Move Theorem~\cite[Theorem~4.1]{amr2}
and the Switch Move Gluing Operation~\cite[Theorem~5.1]{amr2}
from our previous study of modifications of hyperbolic 3-manifolds that are link complements; see
Theorem~\ref{switch} for the statement of the Switch Move Theorem.
More specifically, for each admissible surface $S$, we apply these theorems
to construct a finite volume hyperbolic 3-manifold $N(S)$
with an order-2 isometry whose fixed point set is two-sided and contains
a component $\S$ diffeomorphic to $S$; see Theorem~\ref{thmtg}
for additional topological properties satisfied
by $N(S)$ and $\S$. In Section~\ref{seccmc},
we apply geometric arguments to prove that
for any $T>0$, there is a finite cover $N_T(S)$ of $N(S)$ together with a
lift $\wt{\S}$ of $\S$, so that one of the two $T$-parallel
surfaces to $\wt{\S}$ in $N_T(S)$ is a properly embedded, totally umbilic surface
diffeomorphic to $S$ with
mean curvature $\tanh(T)$;
crucial in these arguments is the property that the fundamental
group of a hyperbolic 3-manifold of finite volume is LERF (see Definition~\ref{defLERF}).

\section{The proof of Theorem~\ref{proper1}.}

In this section, we explain why a totally umbilic surface
in a hyperbolic 3-manifold of finite volume must be proper
and then show how Theorem~\ref{proper1} follows from this properness property.
In order to carry out these proofs, we need the
following definition, which will also be used in
Section~\ref{seccmc} to construct admissible totally
umbilic surfaces that are parallel to totally
geodesic ones.

\begin{definition}\label{deftparallel}
Let $N$ be a Riemannian 3-manifold and $f\colon S \to N$ be
a two-sided embedding with image $\S$ and
unitary normal
vector field $\eta$. We define, for
$t>0$, the {\em $t$-parallel surface
to $\S$} as the image $\S_t$ of the immersion
\begin{equation*}\label{diff}
f_t\colon S \to N,\quad
x \mapsto {\rm exp}(t\eta(f(x))).
\end{equation*}
Thus,
$$\S_t = \{{\rm exp}(t\eta(p))\mid p\in \S\}.$$
\end{definition}

\begin{proof}[Proof of Theorem~\ref{proper1}]
We first prove item~\ref{itemprop} of the theorem.
Since it is well-known that
totally umbilic surfaces with mean curvature $H\geq 1$
in a hyperbolic 3-manifold of finite volume $N$ are
either flat tori or Klein bottles of constant mean curvature $1$ in cusp ends of
$N$ or geodesic spheres, item~\ref{itemprop} holds for $H\geq1$.

Next, we consider the totally geodesic case. For a given manifold
$M$, let $\wh{T}(M)$ denote the bundle of unoriented tangent two-planes of $M$. Then, the
next result follows from the work of Shah~\cite{shah1} (also see
Ratner~\cite{ratn4} and Payne~\cite{pay1}). The phrase
\textquotedblleft immersed surface $f(S)$\textquotedblright \ in a 3-manifold $N$
is used to indicate that
the
image surface $f(S)$
of an immersion $f\colon S\to N$ may have points of self-intersection.

\begin{theorem}\label{thmshah}
Let $f\colon S\to N$ be a complete, totally geodesic
immersion of a surface $S$ to a hyperbolic
3-manifold $N$ of finite volume. Then, either
$f(S)$ is a properly immersed surface of finite area
or $f_*(\wh{T}(S))$ is dense in $\wh{T}(N)$.
\end{theorem}

In fact, Theorem~\ref{thmshah} can be seen to hold
by the following discussion. Let $f\colon S\to N$ be as stated and assume that
$S$ is endowed with a hyperbolic metric.
Then, both $S$ and $N$ are examples of locally symmetric spaces of rank one.
By the last statement of~\cite[Theorem~1.1]{pay1}, the closure of the image
$f(S)$ is a totally geodesic submanifold of $N$, which in the context of this theorem
means $f(S)$ is either proper
or it is dense in $N$. If $f(S)$ is not proper,
then, by~\cite[Theorem~D]{shah1}, $f_*(\wh{T}(S))$ is dense in $\wh{T}(N)$ and the
theorem holds.

Let $\S$ be a totally geodesic surface in a hyperbolic 3-manifold
$N$ of finite volume.
Observe that the closure of
$\S$ in $N$ is a minimal lamination of $N$ of class $C^{\a}$, for all $\a \in (0,1)$.
Therefore, $\wh{T}(\S)$ is not dense in $\wh{T}(N)$ and so
Theorem~\ref{thmshah} implies that $\S$ must have finite area and
be proper in $N$.

To finish the proof of item~\ref{itemprop},
let $\S$ be a totally umbilic
surface in $N$ with mean curvature $H\in (0,1)$.
Assume $\S$ is oriented with respect to a unit normal field $\eta$
pointing towards its mean convex side.
Let $T = \tanh^{-1}(H)>0$ and consider the $T$-parallel immersion
$f_T\colon \S\to N$ with immersed image surface $\S_T$.
Then, $\S_T$ is a complete totally geodesic immersed surface.

We claim that $\S_T$ is proper. Otherwise, $\S_T$ has
infinite area and Theorem~\ref{thmshah}
implies that $\wh{T}(\S_T)$
is dense in $\wh{T}(N)$. Since $\S_T$ has bounded norm of its
second fundamental form, $\S_T$ intersects itself transversely in a dense
set of points in $N$. Let $\Pi\colon \hn3 \to N$
denote the universal Riemannian covering map and
let $\wt{\S}_T^1,\,\wt{\S}_T^2\subset \hn3$
be two components of $\Pi^{-1}(\S_T)$
corresponding to two lifts of $\S_T$ that intersect transversely in $\hn3$.
Let $C_1,\,C_2$ be the respective boundary circles of $\wt{\S}_T^1$,
$\wt{\S}_T^2$, intersecting transversely in the
boundary sphere at infinity of $\hn3$.
Consider two respective lifts $\wt{\S}_1,\,\wt{\S}_2$ of $\S$ in $\hn3$
with the same circles $C_1,\,C_2$ at infinity.
Then, $\wt{\S}_1$ intersects $\wt{\S}_2$ transversely along a proper arc,
which implies that $\S$ is not embedded, a contradiction.
It also follows from this argument that $\S_T$ must be a proper totally
geodesic (embedded) surface in $N$.

To show that $\S$ is proper, there are two cases to consider.
First, assume that $\S_T$ is two-sided and
oriented with respect to the unitary normal field
$\wh{\eta}$ corresponding to the opposite orientation
from the induced immersion $f_T$.
Then, as $\S$ is connected, it is the image of the $T$-parallel
immersion $\wh{f}_T\colon \S_T\to N$,
which must be a proper map,
since the inclusion map
of $\S_T$ in $N$ is proper and the distance between
any two corresponding points $x\in \S_T,\,\wh{f}_T(x)\in \S$
is bounded by $T$.
On the other hand, if $\S_T$ is one-sided,
we may pass to the (proper) two-sheeted, two-sided cover
of $\S_T$ and repeat the same argument,
finishing the proof of item~\ref{itemprop} of
Theorem~\ref{proper1}.

As already observed, a surface $\S$ appears as a
totally umbilic surface
with mean curvature $H_\S\geq 1$ in a hyperbolic manifold $N$
of finite volume if and only if
$\S$ is a geodesic sphere with $H_\S>1$ or it is a flat torus or a flat
Klein bottle in a cusp end of $N$ when $H_\S=1$.
On the other hand, if $\S$ is a totally umbilic surface
with $H_\S\in[0,1)$, which must be proper by item~\ref{itemprop},
then
Corollary~4.7 of~\cite{meramos1} implies
that the Euler characteristic of $\S$ is negative, completing the proof of
items~\ref{itempos} and~\ref{itemzero}.

Note that the main statement of item~\ref{itemneg}
follows from the above discussion and the fact that
any properly immersed, infinite topology surface
with constant mean curvature $H\in[0,1)$
in a hyperbolic 3-manifold of finite volume
has unbounded norm of its second fundamental form;
see item~4 of~\cite[Theorem~1.3]{meramos1} for this unboundedness property.

The other statements of item~\ref{itemneg}
will be explained next. Suppose that $\S$ is a
totally umbilic surface in a hyperbolic
3-manifold $N$ of finite volume, with finite
negative Euler characteristic and $H_\S\in[0,1)$.
Then, item~\ref{itemnega} follows immediately
from~\cite[Corollary~4.7]{meramos1}.

Next, we prove items~\ref{itemnegb} and~\ref{itemnegc} by showing that for any
$t\in(0,T)$, the
$t$-parallel immersion $f_t\colon \S\to N$
is injective, where $T=\tanh^{-1}(H_\S)$.
Recall, from the proof of item~\ref{itemprop}, that
the image surface $\S_T = f_T(\S)$ is a
totally geodesic (embedded) surface in $N$.
Let $\wt{\S}_T\subset \hn3$ be a component of $\Pi^{-1}(\S_T)$
and let $\wt{\S}$ be a component of $\Pi^{-1}(\S)$
with the same boundary circle at
infinity as $\wt{\S}_T$. Let $W\subset \hn3$ be the closed region
with boundary $\wt{\S}\cup \wt{\S}_T$.
Since both $\S$ and $\S_T$ are embedded,
if $\sigma$ is a covering transformation
of $\Pi$, then either $\sigma$ maps $W$ to itself,
in which case $\sigma$ leaves invariant each
surface in $W$ parallel to $\wt{\S}$,
or $\sigma(W)\cap W = \wt{\S}_T$,
in which case $\sigma$ is a glide reflection along $\wt{\S}_T$ or a
loxodromic transformation of $\hn3$ with
respect to a geodesic $\g$ in $\wt{\S}_T$ that has
order-two rotational part about $\g$,
or $\sigma(W)\cap W = \varnothing$.
It follows that $f_t$ is injective for
all $t\in(0,T)$ and $f_T$ is injective
if and only if $\S_T$ is two-sided.
In the case where $\S_T$ is one-sided,
then the induced immersion $f_T\colon\S\to\S_T$ is a double covering of $\S_T$,
and items~\ref{itemnegb} and~\ref{itemnegc}
follow.
\end{proof}

\section{Results on hyperbolic link complements.}

\noindent
\begin{figure}
\noindent \quad
\quad\begin{minipage}{0.42\textwidth}
\begin{center}
\includegraphics[width=0.99\textwidth]{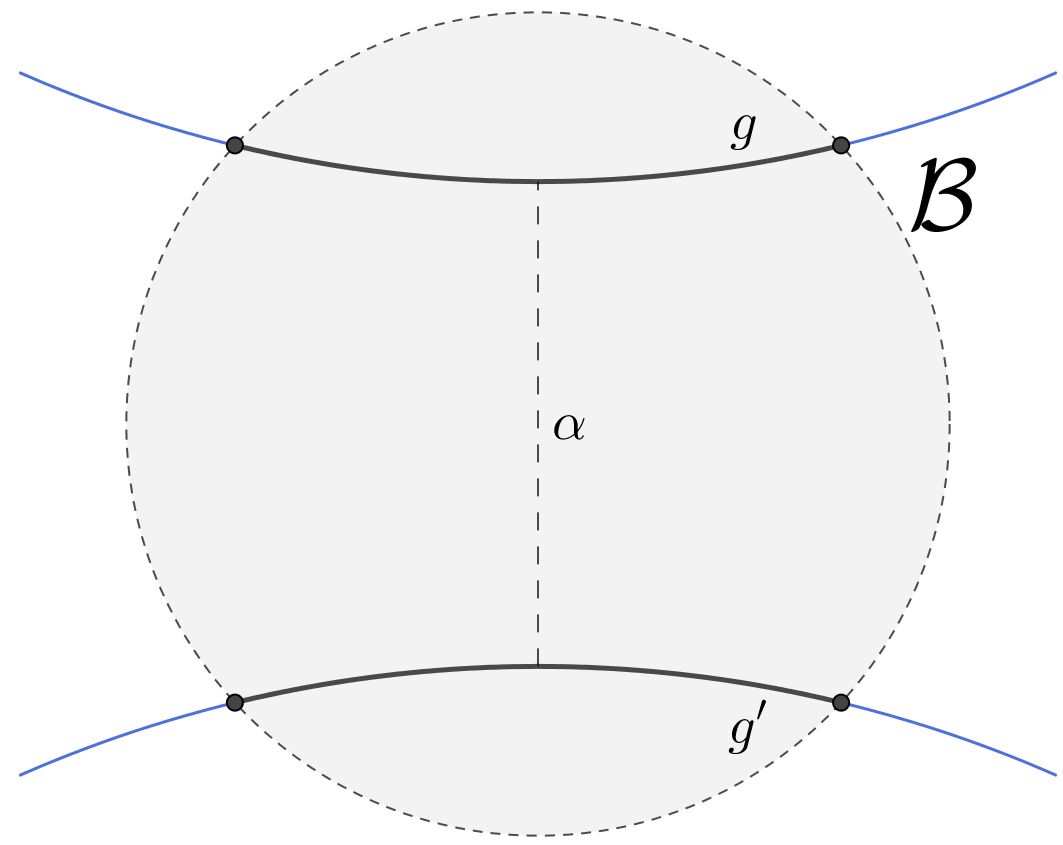}

(a)
\end{center}
\end{minipage}\hfill
$\longrightarrow$\hfill
\begin{minipage}{0.42\textwidth}
\begin{center}
\includegraphics[width=0.99\textwidth]{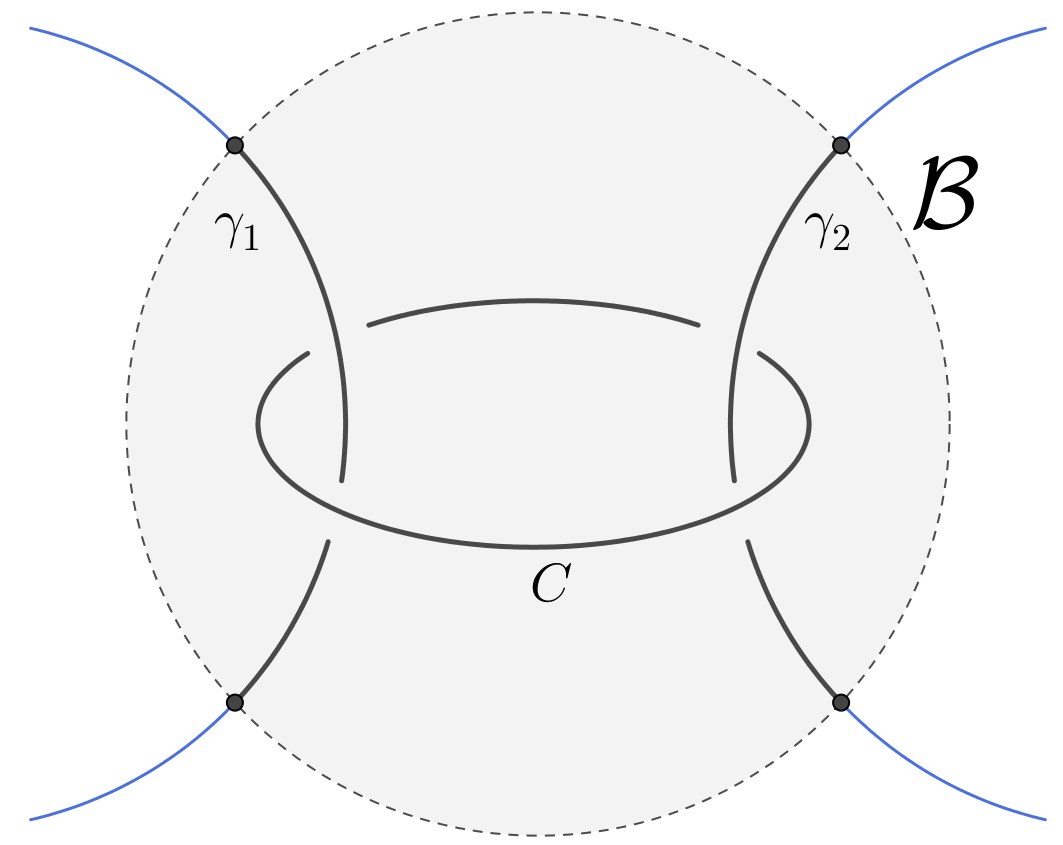}

(b)
\end{center}
\end{minipage}
\quad \quad
\caption{The Switch Move replaces the
arcs $g$ and $g'$ by the tangle $\g_1\cup \g_2\cup C$.
\label{figaugmented}}
\end{figure}

In this section, we present some results concerning the
hyperbolicity of link complements in 3-manifolds
that will be used to construct
the totally geodesic examples described in Theorem~\ref{thmtg}.

\begin{theorem}[{Switch Move Theorem~\cite[Theorem~4.1]{amr2}}]\label{switch}
Let $L$ be a link in a 3-manifold $M$
such that $M\setminus L$ admits a complete hyperbolic metric of
finite volume. Let $\a\subset M$ be a
compact arc which intersects $L$ transversely in its two distinct
endpoints, and such that int$(\a)$ is a properly
embedded geodesic of $M\setminus L$. Let
$\cB$ be a closed ball in $M$ containing $\a$ in its interior and
such that
$\cB\cap L$ is composed of two arcs in $L$, as
in Figure~\ref{figaugmented}(a).
Let $L'$ be the resulting link in $M$ obtained by replacing
$L \cap \cB$ by the components as appearing in
Figure~\ref{figaugmented}~(b).
Then $M\setminus L'$ admits a complete hyperbolic metric of
finite volume.
\end{theorem}

A consequence of the Switch Move Theorem
is the following Untwisted Chain Theorem.

\begin{figure}
\noindent \quad
\quad\begin{minipage}{0.42\textwidth}
\begin{center}
\includegraphics[width=0.99\textwidth]{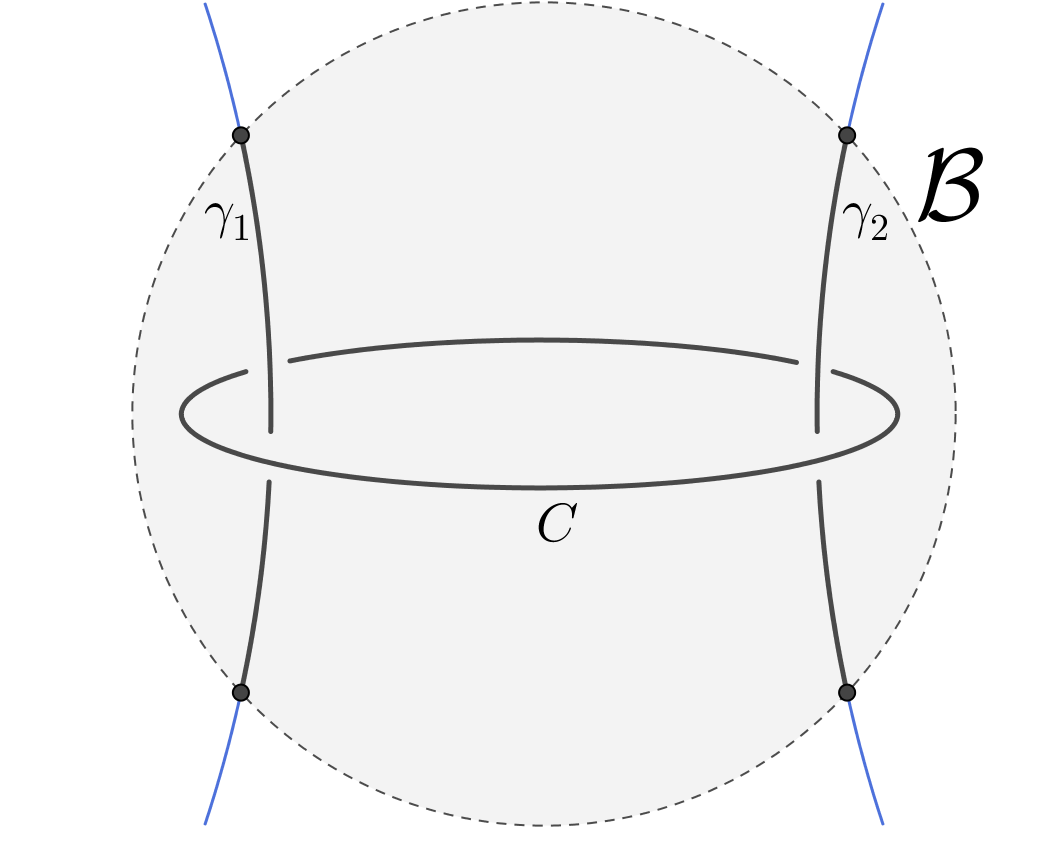}

(a)
\end{center}
\end{minipage}\hfill
$\longrightarrow$\hfill
\begin{minipage}{0.42\textwidth}
\begin{center}
\includegraphics[width=0.99\textwidth]{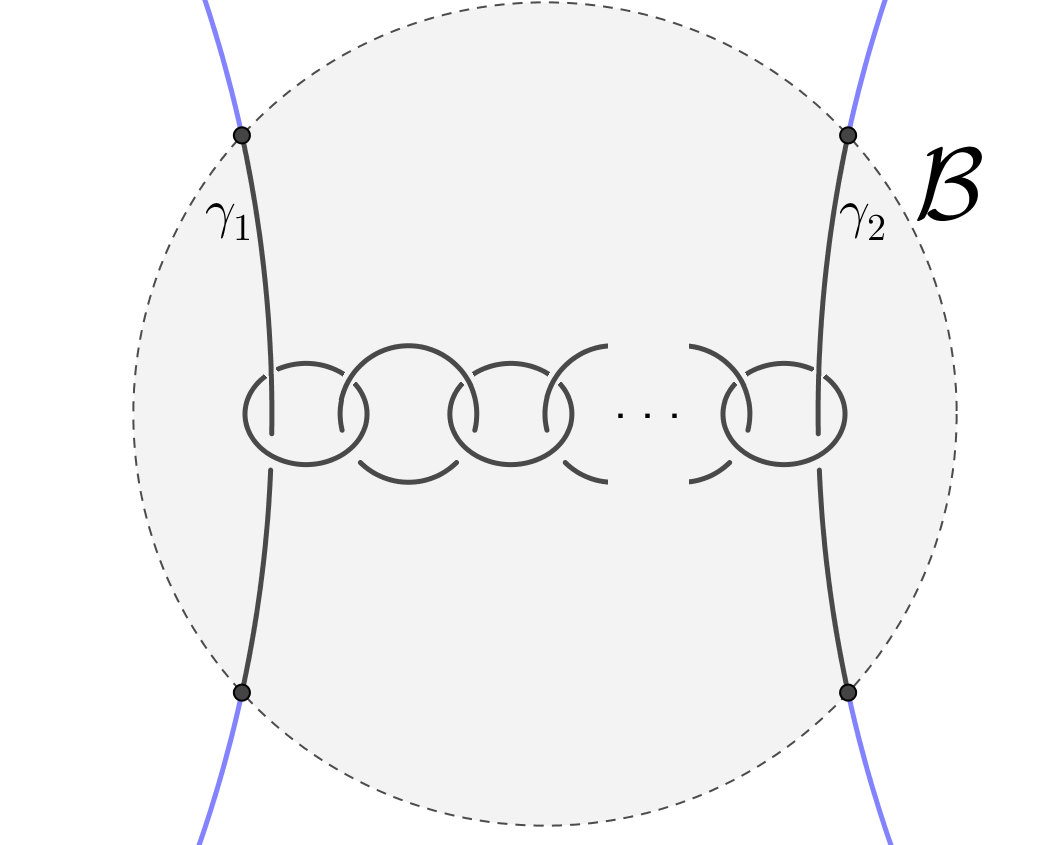}

(b)
\end{center}
\end{minipage}
\quad \quad
\caption{For any positive integer $k$, replacing the trivial component $C$ in
(a) with
the untwisted chain with $2k+1$ components as in (b)
preserves hyperbolicity of the complement.\label{chainlemma1}}
\end{figure}

\begin{corollary}[{Untwisted Chain Theorem}]\label{chain}
Let $L$ be a link in a 3-manifold $M$
such that the link complement $M\setminus L$ admits a
complete hyperbolic metric of finite volume.
Suppose that there is a sphere $\cS$ in $M$ bounding a ball $\cB$
that intersects $L$ as in Figure~\ref{chainlemma1}~(a).
For any positive integer $k$,
let $L'$ be the resulting link in $M$ obtained by replacing
$L \cap \cB$ by the untwisted chain with $2k+1$ components
as in Figure~\ref{chainlemma1}~(b).
Then $M\setminus L'$ admits a complete hyperbolic metric of
finite volume.
\end{corollary}
\begin{proof}
We first prove Corollary~\ref{chain} when $k=1$.
Let $L$, $\cB$ and $M$ be as stated
and let $D\subset \cB$ be a two-punctured disk in $\cB\setminus L$
with $\partial D = C$, where $C$ is the circle component of $\cB\cap L$.
Since $D$ is an incompressible three-punctured sphere in the
hyperbolic 3-manifold $M\setminus L$, then~\cite[Theorem~3.1]{ad1}
gives that, up to isotopy, $D$ is totally geodesic.
Hence, there exists a compact arc $\a\subset \ol{D}$
in the closure of $D$,
such that its interior $\wt{\a}$ is a proper geodesic
in the hyperbolic metric of $M\setminus L$.
Moreover, $\a$ is transverse to $L$ with the endpoints
of $\alpha$ contained on $C$ and
$\a$ separates the punctures of $D$.
After applying the Switch Move Theorem in a neighborhood
of $\alpha$, we obtain a link $L'$ as in Figure~\ref{figchainin},
such that $M\setminus L'$
admits a complete, hyperbolic metric of finite volume.
The general case follows by induction on $k$, as indicated in Figure~\ref{figchainin}.
\begin{figure}
\centering
\includegraphics[width=0.42\textwidth]{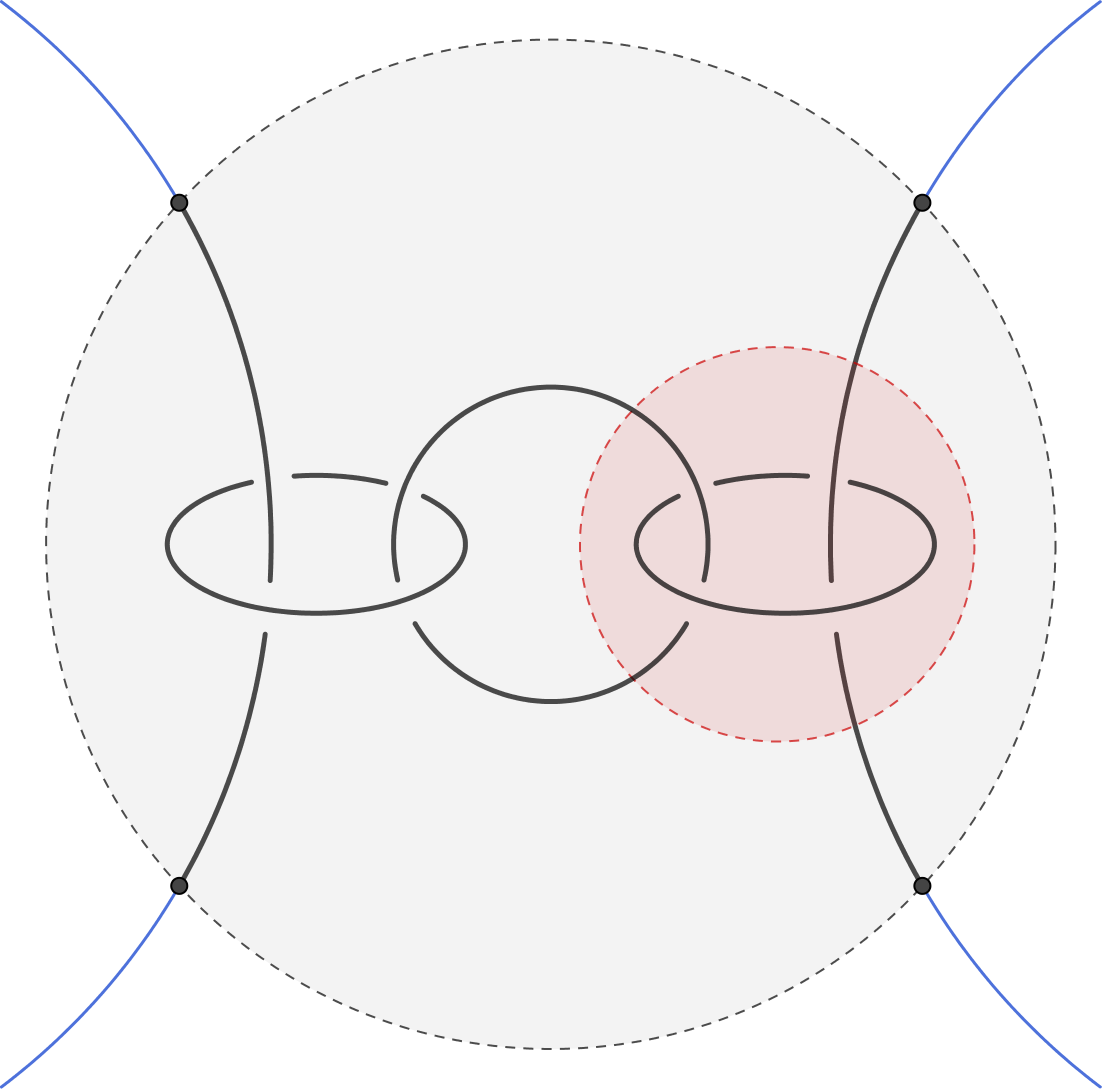}
\caption{The Untwisted Chain Theorem with $k=1$
can be repeated in the highlighted subball to obtain
any odd number of components.\label{figchainin}}
\end{figure}
\end{proof}

\begin{theorem}[Theorem 1.1 of~\cite{Aetal}]
\label{thm:main}
Let ${S}$ be a closed (possibly non-orientable)
surface with $\chi({S})\leq 0$. Then, there exists a link
$L$ in ${S} \times (0,1)$ such that:
\begin{enumerate}[a.]
\item\label{item:a} If $\chi({S})<0$, $(S \times [0,1])\setminus L$
admits a complete hyperbolic metric of finite volume with totally geodesic
boundary.
\item\label{item:b} If $\chi({S})=0$,
$(S \times (0,1))\setminus L$
admits a complete hyperbolic metric of finite volume.
\end{enumerate}
\end{theorem}

\begin{remark}{\rm
In fact, the link $L$ given by Theorem~\ref{thm:main}
can be any {\em fully alternating link}, see~\cite{Aetal}.
A fully alternating link $L$ in $S\times (0,1)$
is a link that admits a
projection to $S$ that is {\em alternating} in the sense that
the link can be oriented so that two consecutive crossings
have distinct over/under signs and is {\em full}
in the sense that every component of the
complement of the projected image
of $L$ on $S$ is a disk.

We also note that Theorem~\ref{thm:main} is proved in~\cite{hp}
in the case
where $S$ is orientable,
without obtaining a totally geodesic boundary when $\chi(S) < 0$.
The proof in~\cite{hp} could be extended to show these additional facts and
it uses different techniques than those applied in~\cite{Aetal}.}
\end{remark}

\section{Construction of hyperbolic 3-manifolds
with totally geodesic surfaces.}\label{sectg}

\begin{theorem}\label{thmtg}
Let $S$ be a surface with
finite negative Euler characteristic.
There exists a finite volume
hyperbolic 3-manifold $N$
and a proper, two-sided embedding $f\colon S\to N$ with totally geodesic image
$\S$. Moreover:
\begin{enumerate}
\item\label{itemclosed}
If $S$ is closed (resp. orientable), $N$ is closed
(resp. orientable).
\item\label{itemends}
If $e_1$ and $e_2$ are
distinct ends of $\S$, then $N$ contains disjoint cusp ends
$\cC_{1},\,\cC_{2}$ such that, for $i=1,2$, \ $\S\cap \cC_{i}$ is an annular
representative of $e_i$.
\item\label{itemisometry}
$\S$ is a two-sided component of the fixed point set of
an order-two isometry of $N$.
\end{enumerate}
\end{theorem}
\begin{proof}
Let $S$ be as stated.
To prove Theorem~\ref{thmtg}, we construct a
complete hyperbolic 3-manifold of finite
volume $N$ together with an order-two
diffeomorphism $\varphi$ that has a two-sided
fixed point set containing a component $\S$ diffeomorphic to $S$.
By the Mostow rigidity theorem,
after changing coordinates
by a diffeomorphism isotopic to the identity,
we may assume that $\varphi$ is an isometry of $N$,
from where it follows that $\S$ is totally geodesic
and item~\ref{itemisometry} holds.

The proof breaks up into cases
which are treated separately.

\begin{case}\label{case1}
$S$ is an $n$-punctured sphere with $n\geq 3$.
\end{case}

\begin{proof}
For $n\geq 3$, let $L_n \subset \rth$
be the $2n$-component untwisted chain link;
the $L_4$ version appears in Figure~\ref{daisy}.
Then, by adding to $\rth$ the point at infinity and
considering $L_n\subset \sn3$,\, $\sn3\setminus L_n$
has an explicit hyperbolic metric
of finite volume,
as described in Example~6.8.7 of~\cite{th1}.

Thinking of the compactified $xy$-plane as a sphere $\cS\subset \esf^3$,
we can view every other component of $L_n$ as being contained in
$\cS$ with the remaining components perpendicular to $\cS$
and symmetric with respect to reflection $R$ through $\cS$.
Then, the restriction of $R$ to the link complement $\sn3\setminus L_n$ is
an order-two isometry of the hyperbolic metric described above,
with two-sided fixed point set
$\cS\setminus L_n$. This fixed point
set contains an $n$-punctured sphere $S_n$,
where the $n$ punctures come from the $n$ components
of $L_n$ in $\cS$. It follows directly that $S_n$ is totally geodesic and
satisfies the statements of the theorem.
\end{proof}

\begin{figure}
\begin{center}
\includegraphics[width=0.68\textwidth]{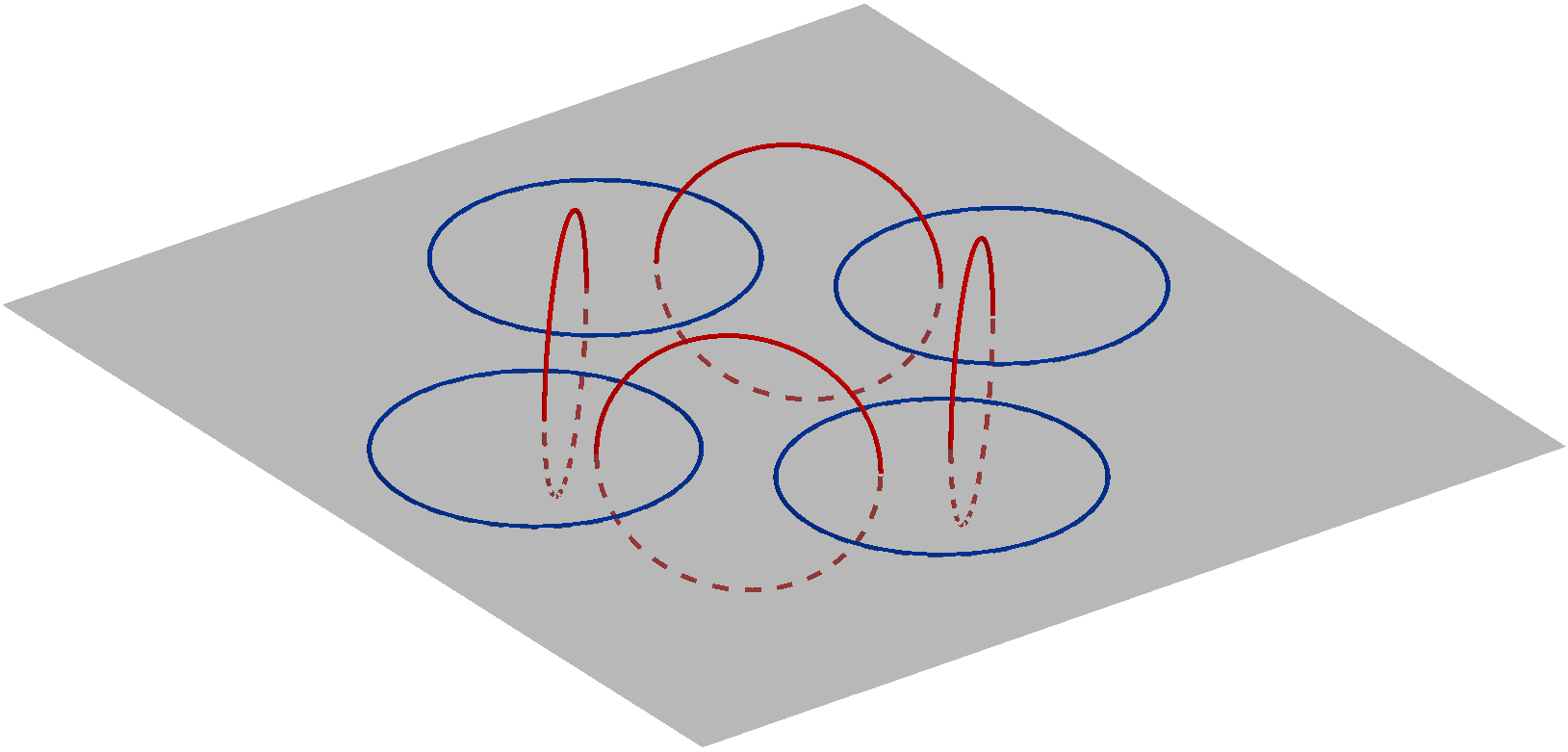}
\caption{The untwisted chain link $L_4$.}
\label{daisy}
\end{center}
\end{figure}

\begin{case}\label{case2}
$S$ is an $n$-punctured projective plane
with $n\geq2$.
\end{case}

\begin{proof}
Fix $n\geq 2$ and let $k = n-1$.
Let $L_{2k}$ be the $4k$-component untwisted chain link in $\sn3$.
As in Case~\ref{case1}, $\cS$ denotes the compactification
of the $xy$-plane and we assume that every other component of $L_{2k}$ is contained
in $\cS$.
Furthermore, we assume that the $2k$ components of $L_{2k}$
that are not contained in $\cS$
are perpendicular to $\cS$ and lie
on the unit sphere $\sn2$ centered at the origin of $\rth$
and that $L_{2k}$ is invariant under the inversion $\phi({\bf x}) = -{\bf x}$
through the origin and the reflection $\Phi$ through $\esf^2$.

Let $Z\subset \sn3$ denote the $z$-axis with the point at infinity.
Define $J_k = L_{2k}\cup Z$ and
let $M = \sn3\setminus J_k$. Then, $M$
is a $2k$-cover of $\sn3\setminus B_3$,
where $B_3$ is the Borromean ring with three components.
Since $\sn3\setminus B_3$ admits a complete hyperbolic metric of finite
volume (see~\cite[Section~3.4]{th1}), there exists a complete hyperbolic
metric $g$ on $M$.
Moreover, $\phi$ and $\Phi$
restrict to isometries of $(M,g)$.

Note that $\sn2\setminus J_k$ contains a connected
component $\wt{S}$ which is a $(2k+2)$-punctured sphere,
where $2k$ punctures come
from $L_{2k}$ and the other two from $\sn2\cap Z$. By construction,
$\Phi\vert_{\wt{S}} = {\rm Id}_{\wt{S}}$, and so $\wt{S}$ is totally geodesic.

Let $N = M/\phi$. Since $\phi$
is a fixed-point free, orientation reversing
order-two isometry of $(M,g)$, $N$ is
a non-orientable hyperbolic manifold that is
double covered by $M$.
Since $\phi\vert_{\sn2}$ is the antipodal map and
$\phi(\wt{S}) = \wt{S}$, the surface
$S = \wt{S}/\phi$ is a $(k+1)$-punctured projective plane in $N$.
Also, since $\Phi$ and $\phi$ commute with each other,
the map $\Phi$ descends to $N$ as an order-two isometry of $N$
which contains $S$ in its fixed point set.
Since $k+1 = n$ and $S$ satisfies
the properties stated by Theorem~\ref{thmtg}, this
proves Case~\ref{case2}.
\end{proof}

\begin{case}\label{case3}
$S$ is closed.
\end{case}
\begin{proof}
For the following construction, see Figure~\ref{figP}~(a).
Consider $\sn1$ to be the unit circle in the $yz$-plane
and let $P= {S}\times \esf^1$. Let $\mathbb{S}^1_+ = \sn1\cap\{z\geq0\}$,
$\mathbb{S}^1_- = \sn1\cap\{z\leq0\}$,
$M_1 = S\times\mathbb{S}^1_+$ and $M_2 = S\times\mathbb{S}^1_-$.
Then, $M_1,\,M_2$ are subsets of $P$
glued along their boundary surfaces
$S_1 = S\times\{(-1,0)\}$
and $S_2 = S\times\{(1,0)\}$.
Let $R\colon P\to P$ be the
reflective symmetry
interchanging $M_1$ with $M_2$;
the fixed point set of $R$ is $S_1 \cup S_2$.

By Theorem~\ref{thm:main}, there exists
a link $L$ in ${\rm int}(M_1)$ such that $M_1\setminus L$
admits a finite volume hyperbolic metric with
totally geodesic boundary $S_1\cup S_2$.
Let
$L'=R(L)\subset {\rm int}(M_2)$ and
$\Gamma=L\cup L'$.
Then, $P\setminus \Gamma$
admits a complete hyperbolic metric $g$ for which $R$ is an isometry
and the surfaces $S_1,S_2$ are
totally geodesic surfaces forming the fixed point set of $R$.

Note that $P\setminus \Gamma$ is orientable if and only
if $S$ is orientable. Moreover, after performing an appropriate
Dehn filling in the ends of $P\setminus \Gamma$ in a symmetric manner
with respect to $R$,
we obtain a closed hyperbolic 3-manifold
$N$, where $S_1$ and $S_2$ are each
as stated in Theorem~\ref{thmtg}.
\end{proof}

\begin{figure}
\begin{minipage}{0.45\textwidth}
\begin{center}
\includegraphics[width=0.9\textwidth]{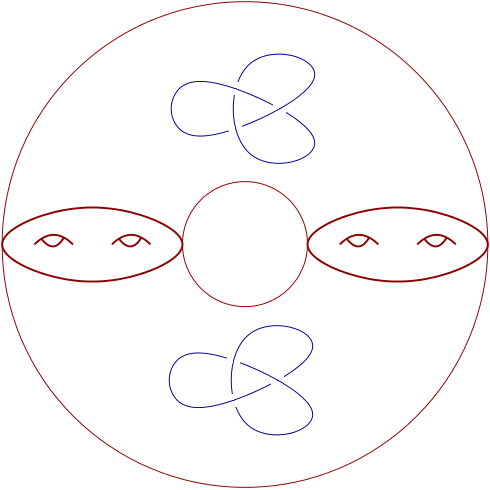}

(a)
\end{center}
\end{minipage}
\begin{minipage}{0.45\textwidth}
\begin{center}
\includegraphics[width=0.9\textwidth]{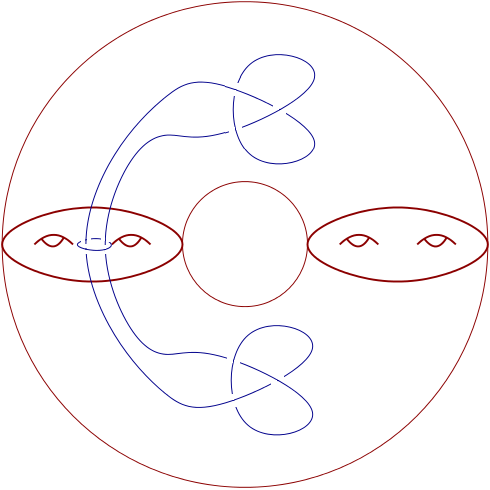}

(b)
\end{center}
\end{minipage}
\caption{(a) is the manifold $P=S\times \sn1$ with
the link $L$ and its reflection $R(L)$. (b)
is the manifold $P=S'\times \sn1$ with the link $L_1$ that
creates a puncture to $S_1$.\label{figP}}
\end{figure}
\begin{case}\label{case4}
$S$ is $S'$ punctured one time, where
$S'$ is a closed surface with $\chi(S')< 0$.
\end{case}
\begin{proof}
The starting point to this case is the closed
manifold $P = S'\times \sn1$.
As in Case~\ref{case3}, there exists a link $\G = L \cup R(L)$,
where $L$ is a link in the interior of
$M_1 = S'\times \mathbb{S}^1_+$, such that
$P\setminus \G$ is hyperbolic,
the reflection $R$ restricts to an isometry
and the fixed point set of $R$ consists of the two
totally geodesic surfaces $S_1 = S'\times \{(-1,0)\}$
and $S_2 = S'\times \{(1,0)\}$.

In the hyperbolic metric of $M_1\setminus L$,
let $\a_1$ be a minimizing geodesic ray from $S_1$ to $L$.
Then, $\a_1$ is proper, perpendicular to $S_1$
and $\a_1$ does not intersect $S_2$.
Let $\a = \a_1\cup R(\a_1)$. Then, $\a$ is
a complete geodesic of $P\setminus \G$
and the closure of $\a$ in $P$ admits a neighborhood
$\cB\subset P$ with the following properties: $R(\cB) = \cB$,
$\cB \cap S_1 = E$ is a disk, $\cB\cap S_2 = \varnothing$ and
$\cB$ intersects $\G$ in
two arcs $g\subset L$ and $g'=R(g)\subset R(L)$. Then, $\cB$ satisfies
the hypothesis of Theorem~\ref{switch}.
Then, we can replace the arcs $g\cup g'$ in $\cB\cap \G$ by
a tangle $\g_1\cup\g_2\cup C$ as in Figure~\ref{figaugmented}~(b)
to form a new link $L_1\subset P$,
where we may choose $C$ as
a circle in $\rm{int}(E)$ that bounds a disk
in $E$, punctured once by $\a$ (see Figure~\ref{figP}~(b)).
Then, Theorem~\ref{switch} gives
that $P\setminus L_1$ admits a complete,
hyperbolic metric of finite volume,
and since we may choose the arcs $\g_1,\,\g_2$ to be invariant under $R$,
it follows that $R$ restricts to an isometry $\varphi$.
Note that the fixed point set of $\varphi$ contains three connected
components, one being $S_2$ and the other two being
the connected components of $S_1\setminus (C\cup \g_1\cup \g_2)$,
one a three-punctured sphere and the other diffeomorphic
to $S$ and satisfying the conclusions required by Theorem~\ref{thmtg}.
\end{proof}

\begin{case}\label{case5}
$S$ is a torus or a Klein bottle punctured once.
\end{case}
\begin{proof}
Let $T$ be either a torus or a Klein bottle and let
$P_1$ be the product manifold $T\times[0,1]$. Let
$\G_1$ be a link in int$(P_1)$ such that
int$(P_1\setminus \G_1)$
is hyperbolic, as given by item~\ref{item:b} of Theorem~\ref{thm:main}.
Let $P = T\times[-1,1]$ and $R\colon P\to P$ be the reflection
$R(x,t) = (x,-t)$; take $P_2$ as the reflected image of $P_1$ in
$P$, with respective hyperbolic link $\G_2 = R(\G_1)\subset P_2$.
Also, let $\a_1$ be a complete geodesic in the hyperbolic metric
of int$(P_1\setminus \G_1)$ with one endpoint in $T\times\{0\}$
and another endpoint in a component $L_1$ of $\G_1$ and let
$\a_2 = R(\a_1)$.

Let $\a\subset P$ be the concatenation of $\a_1$ and $\a_2^{-1}$.
Then, $\a$ is an arc with one endpoint in $L_1$ and another endpoint
in $L_2 = R(L_1)$. Let $\cB$ be a regular
neighborhood of $\a$ in $P$, invariant under $R$ and that intersects
$\G = \G_1\cup \G_2$ in two arcs $g\subset L_1$ and $g' = R(g)\subset L_2$
and intersects $T$ in a disk $\Delta$.
Let $\G'$ be the link in int$(P)$ obtained from
$\G$ by replacing $g\cup g'$ in $\cB$
by the tangle $\g_1\cup \g_2\cup C$
as in Figure~\ref{figaugmented}~(b), where $C\subset \Delta$.
Then, Theorem~5.1 in~\cite{amr2} implies that the interior of the
link complement $P\setminus \G'$ admits a complete hyperbolic
metric of finite volume.
Furthermore, we may choose $\g_1\cup \g_2\cup C$ so that
it is invariant under $R$. Hence, $R$ restricts to an
isometry $\varphi$ of the hyperbolic metric of $P\setminus \G'$.
This proves Case~\ref{case5}, since
the fixed point set of $\varphi$
contains a three-punctured sphere $D$ bounded by $C$
in $T$ and a component $T\setminus \ol{D}$ diffeomorphic to $S$
that satisfies the properties required
by Theorem~\ref{thmtg}.
\end{proof}

The remaining cases to treat in Theorem~\ref{thmtg}
are those where $S$ is diffeomorphic to $S'$ punctured $n\geq 2$ times,
where $S'$ is
a closed surface with $\chi(S')\leq 0$.
If $\chi(S')<0$,
let $P = S'\times \sn1$
with respective hyperbolic link $L_1$ as
given by Case~\ref{case4}.
Let $\cB\subset P$ be the ball where the Switch Move Theorem was
applied. Then, in $P\setminus L_1$, $\cB$ satisfies
the hypothesis of the Untwisted
Chain Theorem, Corollary~\ref{chain}.
Then, we can replace the circle component $C$ of $\cB\cap L_1$
by an untwisted chain with $2n-1$ components as in Figure~\ref{chainlemma1}~(b),
where every other component is in the interior of $D = \cB\cap S_1$
and the remaining
components lie in $\cB$ and are symmetric with respect to $R$ to
create a hyperbolic link $L_n \subset P$.
Again, the reflection $R$ restricts to
an isometry $\varphi$ of the hyperbolic metric of
$P\setminus L_n$, and the fixed point set of $\varphi$ contains
$S_2$, $n-1$ three-punctured spheres and a surface
diffeomorphic to $S$, which finishes the proof
of Theorem~\ref{thmtg} when $\chi(S')<0$. The proof for the case
when $S'$ is a torus or a Klein bottle is analogous.
\end{proof}

\section{The Proof of Theorem~\ref{admiss}.}\label{seccmc}

To finish the proof of Theorem~\ref{admiss}, it suffices
to prove Theorem~\ref{thmcmc} below.

\begin{theorem}\label{thmcmc}
Let $H\in[0,1)$ and suppose that $S$ is a connected
surface of finite topology and negative
Euler characteristic.
Then, there exists a hyperbolic 3-manifold $N$ of
finite volume and a proper, two-sided embedding $f\colon S\to N$, with image $\S$
having mean curvature $H$ and satisfying:

\begin{enumerate}
\item $\S$ is totally umbilic.\label{itema}
\item $\S$ does not separate $N$.\label{itemb}
\item If $S$ is closed (resp. orientable),
then $N$ can be chosen to be closed (resp. orientable).\label{itemc}
\item Each end $e$ of $\S$ admits an annular end representative $E$ which
is embedded in a cusp end $\cC_e$ of $N$.
In addition, if $e$ and $f$ are two
distinct ends of $\S$, the respective cusp ends $\cC_e$, $\cC_f $
are distinct.\label{itemd}
\end{enumerate}
\end{theorem}

To prove Theorem~\ref{thmcmc},
we make use of the totally geodesic examples provided
by Theorem~\ref{thmtg}
to construct hyperbolic manifolds of finite volume
with two-sided, totally umbilic surfaces
which are properly embedded with any admissible finite topology and
mean curvature $H\in (0,1)$. In order to
define these associated totally umbilic examples, we recall,
from Definition~\ref{deftparallel},
that if $N$ is a 3-manifold and $f\colon S\to N$ is a two-sided embedding
with image $\S$ and unitary normal vector field $\eta$,
then the $t$-parallel surface to $\S$ is the image
$\S_t$ of the immersion $f_t\colon S\to N$ defined by
$f_t(x) = \exp(t\eta(f(x)))$.
Note that
the $t$-parallel surface $\S_t$ to
a totally geodesic surface $\S$ in a hyperbolic 3-manifold $N$
is totally umbilic and
has mean curvature $\tanh(t)\in(0,1)$.
For convenience, we will often assume that
the domain of the immersion $f_t$ is
$\S$ instead of the abstract surface $S$.
Note that the ambient distance of two respective points
$x\in\S$ and $f_t(x) \in \S_t$ is uniformly bounded by $t$.

\begin{lemma}\label{proptg}
Let $f\colon S\to N$ satisfy the properties
given by Theorem~\ref{thmtg},
and assume that $N$ has $m\geq 0$ ends.
Then, for any $T>0$:
\begin{enumerate}
\item\label{ii1}
There exists a pairwise disjoint collection of cusp end representatives
of the ends of $N$,
$\{\cC_1(T),\ldots,\cC_m(T)\}$, such that
for any $t\in(0,T]$, $f_t$ is injective on
$f^{-1}(\cup_{i=1}^m\cC_i(T))$.
\item\label{ii2} For $t\in(0,T]$ sufficiently small, the immersions
$f_t\colon S\to N$ are injective.
\item\label{itemgeodesic}
If $\Lambda = \{t>0\mid f_t\text{ is not injective}\} \neq \varnothing$,
then $t_0 = \inf \Lambda>0$ and there exists
a closed geodesic in $N$,
meeting $\S$ orthogonally with length $4t_0$.
\end{enumerate}
\end{lemma}

\begin{proof}
First, we notice that if $S$ is compact, the proofs of
items~\ref{ii1} and~\ref{ii2} are immediate.

Assume now that $S$ is noncompact and let $e_1,\,\ldots,\,e_n$ be
the ends of $\S = f(S)$.
An immediate consequence of item~\ref{itemends} of Theorem~\ref{thmtg} is that there exist pairwise
disjoint annular end representatives $E_1,\,\ldots,\,E_n \subset \S$
and a collection $\{\cC_1,\,\ldots,\,\cC_m\}$
of pairwise disjoint cusp ends of $N$ such that, after possibly passing
to subends, it holds, for $1\leq i\leq n$, that $\S\cap\cC_i =E_i$ with
$\partial \cC_i$ intersecting $\S$ orthogonally and,
for $n+1\leq j\leq m$, $\cC_j\cap \S = \varnothing$.

For a given $i\in\{1,\ldots, n\}$, consider the cusp end $\cC_i$. There exists
a family of compact surfaces $\{\cT_i(s)\}_{s\geq 0}$ (if $\cC_i$ is orientable,
each $\cT_i(s)$ is a torus, otherwise each $\cT_i(s)$ is a Klein bottle)
arising from the descent
of parallel horospheres of $\hn3$ via the universal covering projection,
parameterized by $s = {\rm dist}(\cT_i(s),\cT_i(0) = \partial \cC_i)$
such that $\cup_{s\geq0} \cT_i(s) = \cC_i$.
For any $\lambda > 0$, let
\begin{equation}\label{eqnotation}
\cC_i(\lambda) = \cup_{s\geq \lambda} \cT_i(s),\quad
N(\lambda) = N\setminus \left(\cup_{i=1}^n\cC_i(\lambda)\right),
\end{equation}
and
\begin{equation}\label{eqnotation2}
E_i(\lambda) =E_i \cap \cC_i(\lambda),\quad
\S(\lambda) = \S\cap N(\lambda).
\end{equation}

The assumption that $\S\cap \cC_i = E_i$, for $1\leq i\leq n$,
implies that $E_i(\lambda) = \S \cap \cC_i(\lambda)$. In particular,
for every $\lambda > 0$, we can express $\S$ as a disjoint union
$\S = \left(\cup_{i=1}^n E_i(\lambda)\right)\cup \S(\lambda)$.

Fix $T>0$, $t\in(0,T]$ and
$i\in\{1,\,2,\,\ldots,\,n\}$.
Using the universal covering of $\cC_i$ by a horoball
of $\hn3$, it is easy to see that, for any $s> T$, $E_i$ is orthogonal to $\cT_i(s)$ and
$E_i\cap \cT_i(s)$ is a geodesic
of $\cT_i(s)$, which is injectively mapped by $f_t$ to
a geodesic of $\cT_i(h_t(s))$, where $h_t(s)\in (s-t,s)$ and
the function $h_t$ is increasing;
hence $f_t\vert_{E_i(T)}$ is injective.
In order to finish the proof of item~\ref{ii1}, just note that
$f_t(E_i(T))\subset \cC_i$ and, if $i\neq j$,
$\cC_i\cap \cC_j = \varnothing$.

Next, we prove the second statement of the lemma.
Since $\ol{\S(3T)}$ is compact,
then there exists some
$\ve(T)\in(0,T)$ such that $f_t\vert_{\S(3T)}$ is
injective for all $t\in(0,\ve(T))$. We claim that
$f_t$ is injective for all $t\in(0,\ve(T))$.

Assume that $f_t(x) = f_t(y) = p$, for some $x, y\in \S$ and $t\in(0,\ve(T))$.
If $x\neq y$,
since $f_t\vert_{\S(3T)}$ is injective and
$\{E_1(3T),\,E_2(3T),\,\ldots,\,E_n(3T),\,\S(3T)\}$ is a
partition of $\S$,
there exists some $i\in\{1,2,\ldots, n\}$ such that either
$x\in E_i(3T)$ or $y\in E_i(3T)$.
In particular, $p\in \cC_i(2T)$.
Without loss of generality, assume that $x\in E_i(3T)$.
Item 1 and the fact that $t \in(0,T)$ gives that $f_t\vert_{E_i(T)}$ is injective.
Moreover, $E_i(3T)\subset E_i(T)$, and therefore $y \not\in E_i(T)$.
There are two possibilities: either
$y \in E_j(T)$ for $j\neq i$ or $y \in \S(T)$. If
$y \in E_j(T)$, then $p \in \cC_j$, which is impossible since
$\cC_j\cap\cC_i = \varnothing$. On the other hand, if $y \in \S(T)$, then
$p \in N(2T)$. However, $N(2T)\cap \cC_i(2T) = \varnothing$,
and this contradiction proves item~\ref{ii2}.

We next prove item~\ref{itemgeodesic}.
First, item~\ref{ii2} gives that $t_0 >0$, and
our next argument shows that $f_{t_0}$ is not injective.

Choose $\lambda > t_0$.
By the definition of $t_0$, there exist a sequence
$\{t_k\}_{k\in\N}\subset [t_0,\lambda)$,
$t_k\to t_0$ and points $x_k,\, y_k\in \S$, $x_k \neq y_k$,
such that
$f_{t_k}(x_k) = f_{t_k}(y_k) = p_k$.
From the definition of $f_{t_k}$ and
from the assumption that $t_k<\lambda$ for all $k\in\N$, the triangle inequality
implies that
$d_{N}(x_k,y_k) < 2\lambda$, for all $k\in \N$.

By item~\ref{ii1}, it follows that
$f_{t}\vert_{E_i(\lambda)}$ is injective for all
$i\in \{1,\ldots,n\}$ and all $t\in (0,\lambda)$. Therefore,
$x_k,\,y_k\subset \S(3\lambda)$ for all $k\in\N$, since
$x_k \in E_i(3\lambda)\subset E_i(\lambda)$ gives
that $y_k \in E_i(\lambda)$ and vice versa.
The compactness of $\ol{{\S}(3\lambda)}$ implies that, up to
subsequences, there are $x,y\in \ol{{\S}(3\lambda)}$ such that
$x_k\to x$ and $y_k \to y$. By the continuity
of $(t,z)\mapsto {f}_t(z)$,
it follows that ${f}_{t_0}(x) = {f}_{t_0}(y) = p$.
Moreover,
$\{{f}_t\vert_{\ol{{\S}(3\l)}}\}_{t\in[0,\l]}$
is a smooth, compact family of immersions of the compact
surface $\ol{{\S}(3\l)}$ into $N$. Hence,
there exists an $\ve>0$ such that for
any $t\in[0,\l]$ and any $z \in \ol{{\S}(3\l)}$,
${f}_t\vert_{B_{{\S}}(z,\ve)}$ is injective, which implies that
$d_{{\S}}(x_k,y_k)\geq \ve$, thus
$x\neq y$. This proves that ${f}_{t_0}$ is not
injective.

Let $U\ni x$ and $V\ni y$ be two disjoint open disks of
${\S}$ such that the
restrictions ${f}_{t_0}\vert_U$ and ${f}_{t_0}\vert_V$ are injective.
It follows that ${f}_{t_0}(U)$ and ${f}_{t_0}(V)$
are embedded disks
intersecting at some point $p = f_{t_0}(x) = f_{t_0}(y)$.
Note that the fact that $f_t$
is injective for all $t\in (0,t_0)$ gives
that the intersection of $f_t(U)$ and $f_t(V)$ is tangential at $p$.
Consider the two oriented geodesic rays
$\g_x = \{{f}_t(x)\}_{t\in[0,t_0]}$
and $\g_y = \{{f}_t(y)\}_{t\in[0,t_0]}$. Then the concatenation
$\g$ of $\g_x$ with $\g_y^{-1}$ is a smooth geodesic
arc in $N$, with length $2t_0$ and meeting ${\S}$
orthogonally at the points $x$ and $y$. Finally,
if $\varphi\colon N \to N$ is the order-two
isometry containing $\S$ in its fixed point set,
then $\wt{\g} = \g\cup{\varphi}(\g)$ is
a closed geodesic
with length $4t_0$ and
meeting ${\S}$ orthogonally, proving item~\ref{itemgeodesic}
of Lemma~\ref{proptg}.
\end{proof}

Lemma~\ref{proptg} gives properly embedded, totally umbilic
surfaces for small values of $H$. In order to
finish the proof of Theorem~\ref{thmcmc},
we apply
a technical result stating that the fundamental groups of
hyperbolic 3-manifolds of finite volume satisfy the following definition.

\begin{definition}[Locally Extendable Residually Finite
group]\label{defLERF}
{\em A group $G$ is called }LERF {\em if for every finitely generated
subgroup $K$ of $G$ and any $g \not \in K$, there exists a representation
$\sigma\colon G\to F$ from $G$ to a finite group $F$ such that
$\sigma(g)\not\in \sigma(K)$.
}\end{definition}

The above definition can be
extended as follows.
{\em A group $G$ is LERF
if and only if for every finitely generated subgroup $K$ of $G$
and any finite subset $\cF = \{g_1,\,\ldots,\,g_n\}\subset G$
such that $\cF\cap K = \varnothing$,
there exists a representation $\sigma\colon G \to F$ from $G$ to a finite
group $F$ such that $\sigma(\cF)\cap\sigma(K) = \varnothing$.} Indeed,
if $G$ is LERF and $\cF$ and $K$ are as above,
for each
$i\in\{1,\ldots,n\}$, there exists a representation
$\sigma_i\colon G\to F_i$ from $G$ to a finite group $F_i$
such that $\sigma_i(g_i)\not\in\sigma_i(K)$. Let
$\sigma = \sigma_1\times \ldots\times \sigma_n\colon
G \to F = F_1\times\ldots\times F_n$, then $\sigma$ is a representation
as claimed.
This equivalent definition will be used in the proof of
Theorem~\ref{thmcmc}.

By a series of recent works (see~\cite{afwilton1} for a complete list of
appropriate references) on group theoretical properties of fundamental
groups of hyperbolic 3-manifolds, one has the following result.

\begin{theorem}\label{lemmaLERF}
If $N$ is a hyperbolic 3-manifold of finite volume,
then $\pi_1(N)$ is LERF.
\end{theorem}

The case where $N$ is orientable is
treated in~{\cite[Corollary~4.2.3]{afwilton1}}.
However,
from the discussion in
the book~\cite{afwilton1}, it was not clear to us if
Theorem~\ref{lemmaLERF} applies to the non-orientable case.
For this reason,
we next explain how this property follows from the orientable
case. Recall, from~\cite{afwilton1}, that
a group $G$ is LERF if and only if
any finitely generated subgroup $K$ of
$G$ is closed in the
profinite topology of $G$.
Let $N$ be a non-orientable hyperbolic 3-manifold of finite volume and let
$\wt{N}$ be its oriented 2-sheeted cover. Then $\pi_1(\wt{N})$ is
LERF. Since $\pi_1(\wt{N})$
can be viewed as a finitely generated, index-2
subgroup of $\pi_1(N)$, $\pi_1(\wt{N})$ is closed in the profinite
topology of $\pi_1(N)$.
Let $K \subset \pi_1(N)$ be a finitely generated subgroup.
There are two cases to consider: either $K\subset \pi_1(\wt{N})$,
in which case $K$ is closed in $\pi_1(N)$ since $\pi_1(\wt{N})$ is
LERF and closed in $\pi_1(N)$, or $K\not\subset \pi_1(\wt{N})$,
and there exists some $a \in K$,
$a\not \in \pi_1(\wt{N})$;
it then follows that
$K = \left(K\cap \pi_1(\wt{N})\right)\cup
a\left(K\cap \pi_1(\wt{N})\right)$ is the union of two closed sets,
thus $K$ is closed in $\pi_1({N})$. Therefore,
$\pi_1(N)$ is LERF.

With the above discussion in mind, we now continue with the proof of
Theorem~\ref{thmcmc}. Fix a surface $S$ of
finite topology, with negative Euler characteristic and $T>0$.
Our goal (precisely stated in Lemma~\ref{lemmaembedded} below)
is to produce a hyperbolic 3-manifold $N_{T}$, together
with a two-sided, proper embedding $\wt{f}\colon S \to N_T$
with totally geodesic image
$\wt{\Sigma}$ such that, for each $t\in (0,T]$,
the related parallel immersion
$\wt{f}_t\colon S \to N_T$ is injective.
Since the image surface $\wt{\S}_t = \wt{f}_t(S)$ is totally
umbilic with mean curvature $H = \tanh(t)$,
$\lim_{T\to \infty}\tanh(T) = 1$ and $T$ is arbitrary,
Lemma~\ref{lemmaembedded} proves Theorem~\ref{thmcmc}.

By Theorem~\ref{thmtg}, there exists a hyperbolic 3-manifold of
finite volume $N$, together with a two-sided proper embedding
$f\colon S\to N$, with totally geodesic image $\S$ and
an order-two isometry $\varphi\colon N\to N$ such that
$\varphi\vert_{\S} = {\rm Id}_\S$. Using this particular example,
we prove next lemma.

\begin{lemma}\label{lemmaembedded}
For each $T > 0$, there exists a finite Riemannian covering space
$\Pi_T\colon N_T\to N$ satisfying:
\begin{enumerate}
\item The embedding $f\colon S\to N$ lifts to an embedding
$\wt{f} \colon S\to N_T$.
\item The order-two isometry $\varphi\colon N\to N$
lifts to an order-two isometry $\wt{\varphi}\colon N_T\to N_T$ such that
$\wt{\varphi}\vert_{\wt{\S}} = {\rm Id}_{\wt{\S}}$, where
$\wt{\S} = \wt{f}(S)$.
\item The $t$-parallel surfaces to $\wt{\S}$ are embedded
for all $t\in(0,T]$.
\end{enumerate}
\end{lemma}
\begin{proof}
Fix $T>0$.
Since $N$ is a hyperbolic 3-manifold of finite volume, there are only
a finite number of smooth
closed geodesics in $N$ with length less than $L=5T$.
In particular, the (possibly empty)
collection $\cG_L$ of prime, smooth closed geodesics in $N$ that are orthogonal
to $\Sigma\subset N$ at some point and with length less than $L$
is finite.

Note that if $\cG_L$ were empty, then Lemma~\ref{lemmaembedded}
follows directly from item~\ref{itemgeodesic} of Lemma~\ref{proptg},
by letting $N_T = N$, $\Pi_T = {\rm Id}$, $\wt{f}=f$ and
$\wt{\varphi}=\varphi$.
Thus, we next assume that
$$\cG_L=\{\g_1,\g_2, \ldots,\g_{n(L)}\}\neq \varnothing.$$

For each $j\in \{1,\ldots,n(L)\}$,
let $p_j\in \S\cap \g_j$
be a point where $\g_j$ meets $\S$ orthogonally.
The geodesic $\g_j\in \cG_L$ is invariant under the isometry
$\varphi$; furthermore, $\varphi\vert_{\g_j}$ reverses the orientation
of $\g_j$ and fixes the point $p_j$.

Let $i\colon \S\to N$ be the inclusion map. Choosing $p_1$ as
a base point, we let
$i_*\colon \pi_1(\Sigma,p_1) \to \pi_1(N,p_1)$ be the induced
homomorphism on fundamental groups.
We let
$$K_1 = i_*(\pi_1(\S,p_1))\subset \pi_1(N,p_1)$$
be the image of
the finitely generated group $\pi_1(\Sigma,p_1)$.
Since $p_1\in \g_1$,
$\g_1$ can be considered to represent a nontrivial
element $[\g_1]\in \pi_1(N,p_1)$.
Moreover,
since $\S$ is totally geodesic in a hyperbolic 3-manifold,
$i_*$ is injective and,
for any positive integer $l$,
$[\g_1]^l \not\in K_1$.

Fix a positive integer $k$ sufficiently large so that
\begin{equation}\label{eq:length}
k\cdot \mbox{Length}(\g_1)\geq L
\end{equation}
and let $\cF= \{[\g_1],[\g_1]^2,\ldots,[\g_1]^k\}$. Then,
$\cF\cap K_1 = \varnothing$ and, since $\pi_1(N,p_1)$ is LERF, there exists
a representation $\sigma\colon\pi_1(N,p_1) \to F_1$, from $\pi_1(N,p_1)$
to a finite group $F_1$,
such that $\sigma(\cF) \cap \sigma(K_1)= \varnothing$.
Let $\wh{K}_1,\,\wt{K}_1$
be the subgroups of $\pi_1(N,p_1) $ defined by
$$\wh{K}_1=\sigma^{-1}(\sigma(K_1)),\quad
\wt{K}_1=\wh{K}_1\cap \varphi_*(\wh{K}_1).$$
Note that $K_1\subset \wt{K}_1$, since $K_1\subset \wh{K}_1$ and
$\varphi_*$ fixes all elements of $K_1$.
Also, for $l\in\{1,2,\ldots, k\}$, $[\g_1]^l\not \in \wh{K}_1$,
hence $[\g_1]^l\not \in \wt{K}_1$.

Next we show that $\wt{K}_1$ has finite index in $\pi_1(N,p_1)$.
Since $\wh{K}_1\supset{\rm ker}(\sigma)$ and $F_1$ is finite,
the index of $\wh{K}_1$ in $\pi_1(N,p_1)$ is finite. Moreover,
$\varphi_*\colon \pi_1(N,p_1)\to \pi_1(N,p_1)$ is a group isomorphism;
hence the index of $\varphi_*(\wh{K}_1)$ in $\pi_1(N,p_1)$
is also finite.
Then, since the intersection of two subgroups of finite index also has finite index,
the claim follows.

Let $\Pi_1\colon (N_1,q_1) \to (N,p_1)$ be the Riemannian covering space
of $(N,p_1)$ with image subgroup $(\Pi_1)_*(\pi_1(N_1,q_1))=\wt{K}_1$.
Note that $\Pi_1$ is a finite covering, since the index
of $\wt{K}_1$ in $\pi_1(N,p_1)$ is finite; in particular,
$N_1$ is a hyperbolic manifold of finite volume and, if $N$ is closed
(resp. orientable),
$N_1$ is also closed (resp. orientable).

Since $\varphi_*(\wt{K}_1)= \wt{K}_1$, then, by the lifting criterion,
the maps $i\colon \S\to N$, $\varphi\colon N\to N$
have respective lifts
$$i_1\colon (\S,p_1)\to (N_1,q_1),\quad \varphi_1\colon (N_1,q_1)\to (N_1,q_1).$$
Let $\S_1$ denote the embedded, totally geodesic image surface
of the injective immersion
$i_1$, and note
that $\S_1$ is two-sided and
contained in the fixed point set of the order-two
isometry $\varphi_1$.

Consider the (possibly empty) collection
$\cG_L^1=\{\g_1^1, \g^1_2,\ldots, \g^1_{n_1(L)}\}$
of smooth, prime closed geodesics in $N_1$
that have length less than $L$ and are
orthogonal to $\S_1$ at some point.
Then, $\Pi_1$ induces an injective map
from $\cG_L^1$ to $\cG_L$,
hence $n_1(L)\leq n(L)$. We next prove that
this construction yields $n_1(L)<n(L)$.

\begin{claim}\label{claiminduction}
The image set of geodesics
$\Pi_1(\cG_L^1)=\{\Pi_1(\g_j^1) \mid j=1, \ldots, n_1(L)\}$
forms a subset of $\{\g_1,\g_2,\ldots, \g_{n(L)}\}$,
which does not include $\g_1$.
In particular, $n_1(L)<n(L)$.
\end{claim}
\begin{proof}
Arguing by contradiction,
suppose, after possibly reordering, that $\g_1=\Pi_1(\g^1_1)$ as a set.
Then, $\g_1^1$ is the lift of a certain
smallest power $J$ of $\g_1$, which implies that the length of $\g_1^1$
is equal to $J\cdot \mbox{Length}(\g_1)$. However, for $l\in\{1,\,\ldots,\,k\}$,
$[\g_1]^l\not \in \wt{K}_1$; hence,
none of the powers of $\g_1$ less than or equal to $k$ lift.
Then $J>k$,
and~\eqref{eq:length} implies that the length of
$\g^1_1$ is larger than $L$, which is a contradiction.
\end{proof}

By induction,
Claim~\ref{claiminduction} allows us
to produce a finite Riemannian
cover $\Pi_T\colon N_T\to N$ such that:
\begin{enumerate}
\item The embedding $f\colon S\to N$ lifts to an embedding
$\wt{f}\colon S \to N_T$, with totally geodesic image surface $\wt{\S}$.
\item There are no smooth closed geodesics in $N_T$
with length less than $L$ and intersecting $\wt{\S}$
orthogonally at some point.
\item $\varphi\colon N\to N$ lifts to an order-two isometry
$\wt{\varphi}\colon N_T\to N_T$ and
$\wt{\varphi}\vert_{\wt{\S}} = {\rm Id}_{\wt{\S}}$.
\end{enumerate}

Note that
$\wt{f}$ and $\wt{\varphi}$ satisfy the hypothesis of
Lemma~\ref{proptg}. In particular, since
$\wt{\S}$ was
constructed in such a way that there are no
smooth closed geodesics in $N_T$ with length less than
$L = 5T$ intersecting $\wt{\S}$ orthogonally, item~\ref{itemgeodesic}
of Lemma~\ref{proptg} implies that
the $t$-parallel surfaces to $\wt{\S}$ are embedded for
all $t\in (0,T]$,
which completes the proof of
Lemma~\ref{lemmaembedded}.
\end{proof}

As already explained, Theorem~\ref{thmcmc} now follows from
Theorem~\ref{thmtg} and Lemma~\ref{lemmaembedded},
since in the above construction
the $t$-parallel surface $\wt{\S}_t= \wt{f}_t(S)$
is a properly embedded surface in $N_T$ satisfying
the conditions~\ref{itema},~\ref{itemb},~\ref{itemc} and~\ref{itemd}
of Theorem~\ref{thmcmc}
for $H =\tanh(t)\in(0,\tanh(T)]$.

\nocite{chr2}

\bibliographystyle{plain}
\bibliography{bill}

\center{Colin Adams at cadams@williams.edu \\
Department of Mathematics and Statistics, Williams College, Williamstown, MA 01267}
\center{William H. Meeks, III at  profmeeks@gmail.com\\
Mathematics Department, University of Massachusetts, Amherst, MA 01003}
\center{Álvaro K. Ramos at alvaro.ramos@ufrgs.br \\
Departmento de Matemática Pura e Aplicada, Universidade Federal do Rio Grande
do Sul, Brazil}

\end{document}